\documentclass[11pt]{article}

\usepackage[margin=1in]{geometry}
\usepackage{amsmath,amssymb,amsthm,bbm}
\usepackage{mathtools}
\usepackage{authblk}

\newcommand{\R}{\mathbb R}
\newcommand{\E}{\mathbb E}

\newcommand{\var}{\mathrm{var}}
\newcommand{\cov}{\mathrm{cov}}
\newcommand{\tr}{\mathrm{tr}}
\newcommand{\e}{\mathrm e}

\newcommand{\dt}[1]{\frac{\partial #1}{\partial t}}

\newcommand{\dive}{\mathop{\rm div}\nolimits}
\newcommand{\comment}[1]{}

\DeclareMathOperator{\Ent}{Ent}
\DeclarePairedDelimiter{\abs}{\lvert}{\rvert}
\DeclarePairedDelimiter\norm{\lVert}{\rVert}

\newtheorem{thm}{Theorem}
\newtheorem{lem}[thm]{Lemma}

\theoremstyle{remark}

\theoremstyle{definition}

\newenvironment{rmk}{\textit{Remark :}}{}

\author{Pierre Bizeul}
\affil{Institut de Mathématiques de Jussieu-Paris Rive Gauche}
\title{Entropy and Information jump for log-concave vectors}

\begin{document}
\maketitle
\begin{abstract}
    We extend the result of Ball and Nguyen on the jump of entropy under convolution for log-concave random vectors. We show that the result holds for any pair of vectors (not necessarily identically distributed) and that a similar inequality holds for the Fisher information, thus providing a quantitative Blachmann-Stam inequality 
\end{abstract}

\section{Introduction}

Let $X$ be a random vector distributed according to a measure $\mu$ in $\R^d$, with density $f$ with respect to the Lebesgue measure. We will denote this by $X\sim \mu = fd\lambda$. If $\int f\vert\log f\vert < +\infty$, we define its entropy by
$$\Ent_L(\mu) = \Ent_L(X) = - \int_{\R^d} f\log f$$,
where the subscript $L$ stands for "Lebesgue".

It should be noted that this entropy can be either positive or negative, and that for any invertible matrix $A$, $\Ent_L(AX) = \Ent_L(X) + \log (\abs{\det A})$. Entropy is also translation invariant, and it is classical that, when the covariance matrix is fixed, the Gaussian distribution  maximizes entropy.  It will be useful to normalize vectors so that they are centered, and have covariance matrix identity. Such a vector, as well as its distribution, is called \textit{isotropic}.

The classical Shannon-Stam inequality asserts that taking a convolution increases entropy: for two iid random vectors $X_1$ and $X_2$
$$\Ent_L(X_1)\leq \Ent_L(\frac{X_1+X_2}{\sqrt{2}}).$$
Moreover, there is equality if and only if $X$ has a Gaussian distribution. Now, one can wonder if that equality case is stable, meaning if the \textit{entropy jump} $\Ent_L(\frac{X_1+X_2}{\sqrt{2}}) - \Ent_L(X) $ is small, does it imply that $X$ is almost Gaussian ? The general answer is no, as one can convince himself by considering a well chosen double-bump Gaussian \cite{courtade2018quantitative}.


However, when the distribution of $X$ admits a spectral gap, excluding double-bumped type distributions, some positive answers exist. Recall that $X$ is said to have a spectral gap, or equivalently satisfy a Poincaré inequality, if there exists a constant $c>0$ such that for any smooth enough function $f$, the variance of $f(X)$ can be controlled in terms of the euclidean norm of $\nabla f(x)$ as follows:
$$\var{f(X)} \leq c \, \E \left[ \, \vert\nabla f(X)\vert^2  \, \right] .$$
The smallest such constant $c$ will be denoted $c_X$ and called the Poincaré constant of $X$. Under a spectral gap assumption, it was proven by Ball, Barthe and Naor in \cite{ball2003entropy} that for a one dimensional isotropic random variable $X$,
$$\Ent_L\left(\frac{X_1+X_2}{\sqrt{2}}\right) - \Ent_L(X_1) \geq \frac{1}{2(1+c_X)}(\Ent_L(G) - \Ent_L(X_1))  $$,
where $G$ is a standard Gaussian.

We can rewrite the right-hand side as a Kullback-Leibler divergence. Recall that, if $X$ is isotropic,
$$\Ent_L(G) - \Ent_L(X) = D(X\vert\vert G) = \int_{\R^d}f_\gamma \log(f_\gamma)d\gamma \geq 0 $$,
where $f_\gamma$ is the relative density of $X$ with respect to the Gaussian measure $\gamma$, that is $X\sim f_\gamma d\gamma$ and $G\sim \gamma$. In the sequel we shall use the notation $D(X) = D(X\vert\vert G)$. This is a strong measure of closeness to the Gaussian; for instance the Pinsker-Csiszar-Kullback inequality states that 
$$\left(\int_{\R^d}\vert f - g\vert\right)^2 \leq \frac{1}{2} D(X\vert\vert G)$$,
where $f$ and $g$ are the density of $X$ and $G$, respectively.
In 2012, Ball and Nguyen generalized the result to arbitrary dimension, assuming log-concavity of $X$ (\cite{ball2012entropy}). They use a semigroup approach, differentiating twice the entropy along the Ornstein-Uhlenbeck semigroup.

The Fisher information of a random vector with smooth density $f$ is

$$I_L(X) =\int \frac{\abs{\nabla f}^2}{f} = 4 \int \abs{\nabla\left(\sqrt{f}\right)}^2  ,$$
whenever those integrals are well defined. As for the entropy, for a fixed covariance matrix, the Gaussian is extremal; in this case, it has the smallest information. Information is classically the derivative of the entropy along the semi-group. In the spirit of the Shannon-Stam Inequality, the Blachman-Stamn inequality asserts that taking a convolution decreases the information:

$$I_L(X_1) \geq I_L\left(\frac{X_1 + X_2}{\sqrt 2}\right). $$

As before, we can define a relative information, notably to the Gaussian measure $d\gamma$. If $X \sim f_\gamma d\gamma$ is a random vector with smooth density, we will denote its relative information to $d\gamma$ by

$$I(X\vert\vert G) = I_\gamma(X) = \int \frac{\abs{\nabla f_\gamma}^2}{f_\gamma} d\gamma. $$ 
When $X$ is isotropic, integrating by parts yields:

$$I(X\vert\vert G) = I_L(X) - d = I_L(X) - I_L(G) .$$

Consequently, for a measure $\mu$ on $\R^d$, we write $D(\mu\vert\vert \gamma) = D(X\vert\vert G), \  I_L(\mu) = I_L(X)$ and $I(\mu\vert\vert \gamma) = I(X\vert\vert G)$ where $X\sim\mu$ is a random vector distributed according to $\mu$

In this note, we use the same strategy as in \cite{ball2012entropy}, but improve their result in two directions. First we generalize it to non identically distributed pairs of random vectors. For two measures $\mu$ and $\nu$ define

$$\delta_{E,\lambda}(\mu,\nu) = \Ent_L ( X_\lambda ) -(1-\lambda ) \Ent_L ( X_0 ) 
 - \lambda \Ent_L ( X_1 )   $$,
where $X_0$ and $X_1$ are independent random vectors distributed according to $\mu$ and $\nu$ respectively, and $X_\lambda = \sqrt{1-\lambda}X_0 + \sqrt{\lambda}X_1$. The Shannon-Stam inequality asserts that $\delta_\lambda(\mu,\nu) \geq 0$ and this quantity is precisely the deficit in the Shannon-Stam inequality.

\begin{thm}[Quantitative Shannon-Stam]{\label{thm:qss}}
Let $\mu,\nu$ be two log-concave isotropic measures with Poincaré constant respectively $c_0$ and $c_1$, and $\lambda \in [0,1]$. Then,
\[ 
\delta_{E,\lambda}(\mu,\nu)  
 \geq \frac {\lambda(1-\lambda)}{ 4 \max( c_0,c_1)} \left( D(\mu\vert\vert \gamma ) + D(\nu\vert\vert \gamma ) \right)
\]
\end{thm}
This should be compared with a recent result of Eldan and Mikulincer (\cite{DBLP:journals/corr/abs-1903-07140}, Theorem 3). They get a more general result, allowing $\mu$ and $\nu$ to have different covariance matrices, but in the case where $\mu$ and $\nu$ have the same covariance matrix, they get a worst dependence on the Poincaré constant.

Secondly, we get a same kind of inequality for the information, yielding a stability result for the Blachman-Stam inequality. Define this time the information deficit of a pair of measures by

$$\delta_{I,\lambda}(\mu,\nu) = (1-\lambda ) I_L ( X_0 ) 
 + \lambda I_L ( X_1 ) - I_L ( X_\lambda )$$
where $X_0$ and $X_1$ are independent random vectors distributed according to $\mu$ and $\nu$ respectively, and $X_\lambda = \sqrt{1-\lambda}X_0 + \sqrt{\lambda}X_1$

\begin{thm}[Quantitative Blachman-Stam]{\label{thm:qbs}}
Let $\mu,\nu$ be two log-concave isotropic measures with Poincaré constant respectively $c_0$ and $c_1$, and $\lambda \in [0,1]$. Then,
\[ 
\delta_{I,\lambda}(\mu,\nu)   
 \geq \frac {\lambda(1-\lambda)}{ 4 \max( c_0,c_1)} \left( I ( \mu\vert\vert d\gamma ) + I ( \nu\vert\vert d\gamma  ) \right) . 
\]
\end{thm}

In the sequel, quantities computed with respect to the Lebesgue measure have a subscript $``L"$ while quantities that are computed with respect to the Gaussian measure have none.

\section{A lemma of Ball-Nguyen}
Let $X$ be a random vector with smooth density $f = e^{-\psi}$ with respect to the \textbf{Lebesgue} measure. We define $\sigma_L(X)$ to be the random matrix $\sigma_L(X):= \nabla^2(\psi)(X)$ and, we denote

$$K_L(X) = \E \left[\norm{\sigma_L(X)}^2\right] ,$$
where $\norm{.}$ denotes the Hilbert-Schmidt norm on matrices and the subscripts $L$ yet again stands for "Lebesgue", which we take temporarily as the reference measure. Understanding this quantitiy will prove to be important later on, as it will appear in the second derivative of the entropy along the Ornstein-Uhlenbeck semigroup.

We recall a lemma of Ball-Nguyen \cite{ball2012entropy}, for which we provide a simple proof.

\begin{lem}{(Ball-Nguyen)}
Let $X$ be a random vector in $\R^d$ with smooth density, $E\subset\R^d$ be any subspace, $p_E$ be the orthogonal projection onto $E$ and $X_E = p_E(X)$. Then
$$\sigma_L(X_E) \leq p_E \E\left[\sigma_L(X) \vert X_E\right]p_E^*, \quad a.s ,$$
for the partial order on symmetric matrices.
\end{lem}
\begin{proof}
Let $\psi_E : E \xrightarrow[]{} \R$ be such that $X_E \sim e^{-\psi_E(x)}dx$. We have for $x\in E$ $$\psi_E(x) = - \ln \int_{E^\perp} e^{-\psi(x,y)}dy$$
Then, setting $d\nu_x = \frac{e^{-\psi(x,y)}}{\int_{E^\perp} e^{-\psi(x,y)}dy}dy$, a straightforward computation shows that:

$$\forall x \in E,\quad \nabla^2\psi_E(x) = \int_{E^\perp}\nabla^2_{xx}\psi d\nu_x - \cov_{d\nu_x}(\nabla_x\psi) $$
In particular, 
$$\forall x \in E,\quad \nabla^2\psi_E(x) \leq \int_{E^\perp}\nabla^2_{xx}\psi d\nu_x $$
which is the desired result
\end{proof}

If $X_0$ and $X_1$ are two independent random vectors in $\R^d$ and $\lambda \in [0,1]$, applying the previous lemma to the random vector $(X_0,X_1)$ and the projection $p(x,y) =\sqrt{1-\lambda}x + \sqrt{\lambda}y $ yields:
\begin{lem}{\label{lem:ineq_sigma}}
For any independent random vectors $X_0$, $X_1$ in $\R^d$, with smooth densities, and any $\lambda \in [0,1]$
$$\sigma_L(X_\lambda) \leq \E\left[(1-\lambda)\sigma_L(X_0) + \lambda \sigma_L(X_1) \vert X_\lambda\right]\quad  a.s,$$
where $X_\lambda = \sqrt{1-\lambda}X_0 + \sqrt{\lambda} X_1.$
\end{lem}
If $X_0$ and $X_1$ are log-concave and independent, then so is $X_\lambda$, by Prékopa's theorem. Thus $\sigma_L(X_0), \sigma_L(X_1)$ and $\sigma(X_\lambda)$ are positive matrices, so as in Ball-Nguyen's article, the inequality above translates to an inequality on their norm.

\begin{lem}\label{lem:klebesgue}

For any log-concave independent random vectors $X_0$, $X_1$ in $\R^d$, with smooth densities, and any $\lambda \in [0,1]$
$$(1-\lambda)K_L(X_0) + \lambda K_L(X_1) - K_L(X_\lambda) \geq \lambda(1-\lambda)\E\left[\norm{\sigma_L(X_1)-\sigma_L(X_0)}^2\right] $$
\end{lem}
\noindent
\begin{rmk}
In particular, we have :

$$K_L(X_\lambda) \leq (1-\lambda)K_L(X_0) + \lambda K_L(X_1) $$
which can be seen as a second-order Blachman-Stam inequality.

\end{rmk}

\begin{proof}
As explained, the matrices being positive, the inequality in Lemma \ref{lem:ineq_sigma} implies that:
$$\norm{\sigma_L(X_\lambda)} \leq
\norm{\E\left[(1-\lambda)\sigma_L(X_0) + \lambda \sigma_L(X_1) \vert X_\lambda\right]} $$
Taking the expectation of the square and using Jensen's inequality then implies:
\begin{align*}
    K_L(X_\lambda) &\leq \E\left[\norm{(1-\lambda)\sigma_L(X_0) + \lambda\sigma_L(X_1)}^2\right] \\
    & = (1-\lambda)^2K_L(X_0) + \lambda^2K_L(X_1) + 2\lambda(1-\lambda)\E\left[\langle \sigma_L(X_0) , \sigma_L(X_1) \rangle\right] \\
    &= (1-\lambda)K_L(X_0) + \lambda K_L(X_1) - \lambda(1-\lambda)\E\left[\norm{\sigma_L(X_0) - \sigma_L(X_1)}^2\right]
\end{align*}
since $X_0$ and $X_1$ are independent.
\end{proof}
Now we want to translate this result to the Gaussian setting. Assuming that $X$ has density $f_\gamma = e^{-\varphi}$ with respect to the Gaussian measure, we similarly introduce:

$$ \sigma(X):= \sigma_\gamma(X):= \nabla^2(\varphi)(X) $$
$$K(X):= K_\gamma(X) = \E \left[\norm{\sigma(X)}^2\right] $$
$$I(X) = I_\gamma(X) = \E\left[\norm{\nabla\varphi(X)}^2\right]. $$
Note that the definition of $I$ is consistent with the one given in the introduction.
Then Lemma \ref{lem:klebesgue} becomes:
\begin{lem}{\label{lem:L}} For any log-concave independent isotropic random vectors $X_0$, $X_1$ in $\R^d$, with smooth densities, and any $\lambda \in [0,1]$, we have 
\begin{equation}{\label{eq:L}}
    (1-\lambda)M(X_0) + \lambda M(X_1) - M(X_\lambda) \geq \lambda(1-\lambda)\E\left[\norm{\sigma(X_1)-\sigma(X_0)}^2\right] ,
\end{equation}
where $M(X) = K(X) + 2I(X).$
\end{lem}

\begin{proof}
Let $X$ be an isotropic log-concave random vector with density $e^{-\varphi}$ with respect to the Gaussian measure, and $e^{-\psi}$ with respect to the Lebesgue measure. By definition we have:

$$\sigma(X) = \sigma_L(X) - Id.$$
Hence, 
$$K(X) = K_L(X) - 2\E\left[\tr{(\sigma_L(X))}\right] + n. $$
Now, by integration by parts:
\begin{align*}
 \E\left[\tr{(\sigma_L(X))}\right] &=\int_{\R^n}\dive{(\nabla \psi)}(x) e^{-\psi(x)}dx  \\
 &= \int_{\R^n}\nabla \psi(x)\cdot\nabla \psi(x) e^{-\psi(x)}dx \\
 & = \E\left[\norm{\nabla\psi(X)}^2\right] = \E\left[\norm{\nabla\varphi(X) + X}^2\right] = I(X) + n.
\end{align*}
The lemma follows.
\end{proof}

The next lemma provides a lower bound for the right-hand side in the inequality (\ref{eq:L}). 
\begin{lem}{\label{lem:L2}} For any log-concave independent isotropic random vectors $X_0$, $X_1$ in $\R^d$ with smooth densities we have
$$(1-\lambda)M(X_0) + \lambda M(X_1) - M(X_\lambda) \geq \frac{\lambda(1-\lambda)}{2\max (c_0,c_1)}(I(X_0) + K(X_0) +I(X_1)+K(X_1)), $$
where $c_0,c_1$ are the Poincaré constants of $X_0$ and $X_1$, respectively.
\end{lem}
\begin{proof}
We denote by $c_0$ the Poincaré constant of $X_0$. We condition on $X_1$ and apply the Poincaré inequality for $X_0$ to the function $\nabla\varphi_0(X_0) - \nabla^2\varphi_1(X_1)X_0$ which is centered and we use the fact that $X_0$ and $X_1$ are isotropic:
\begin{align*}
    \E\left[ \norm{\nabla^2\varphi_0(X_0) - \nabla^2\varphi_1(X_1)}^2\right] &\geq \frac{1}{c_0} \E\left[\norm{\nabla\varphi_0(X_0) - \nabla^2\varphi_1(X_1)X_0}^2\right] \\
    & = \frac{1}{c_0} \E\left[\norm{\nabla\varphi_0(X_0)}^2 + \norm{\nabla^2\varphi_1(X_1)}\right] \\
    & = \frac{1}{c_0}(I(X_0) + K(X_1)).
\end{align*}
By symmetry we get:
$$\E\left[\norm{\sigma_\gamma(X_1)-\sigma_\gamma(X_0)}^2\right] \geq  \frac{1}{2\max (c_0,c_1)}(I(X_0) + K(X_0) +I(X_1)+K(X_1))$$
Plugging this into Lemma \ref{lem:L} concludes the proof.
\end{proof}

\section{The Ornstein-Uhlenbeck process and proof of the theorems}
Let $X$ be a random vector in $\R^d$ with density $f$ with respect to the \textbf{Gaussian} measure $\gamma$.
Let $L_\gamma$ be the diffusion operator defined by $L_\gamma f(x) = \Delta f(x) - \nabla f(x)\cdot x$. The differential equation associated to $L_\gamma$ is the modified heat equation:

$$\dt{f_t} = L_\gamma f_t \ ;\quad f_0 = f .$$
Its solution $f_t$ is the relative density of the random vector $X_t = e^{-t}X + \sqrt{1-e^{-2t}}G$, where $G$ is a standard Gaussian independent of $X$, with respect to the Gaussian measure. From this description of $X_t$ it is clear that $X_t$ has a $\mathcal{C}^\infty$ density, with integrability properties as good as $f$, and that the process commutes with convolutions, in the sense that, with the notations of the previous sections, for all independent $X_0$, $X_1$ random vectors, and any $\lambda \in [0,1]$, we have
\begin{equation}{\label{eq:commute}}
    (X_\lambda)_t = (X_t)_\lambda \quad \textrm{in law.}
\end{equation}
It is also useful to note that if $X$ is such that $c_X \geq 1$, in particular if $X$ is isotropic, then for all $t\geq0$:
$$c_{X_t} \leq c_X. $$
Indeed, if $X$ and $Y$ are independent random vectors satisfying a Poincaré inequality and $\lambda \in [0,1]$, then using the conditional variance formula, a few computations show that $$c_{\sqrt{\lambda}X + \sqrt{1-\lambda}Y} \leq \lambda c_X +(1 - \lambda)c_Y, $$ which in our case yields $c_{X_t} \leq e^{-2t}c_X + (1-e^{-2t})\times1 \leq c_X$.

We denote by $(P_t)_{t\geq0}$ the semi-group defined by $P_t(f) = f_t$. The following computations are standard (see \cite{bakry2013analysis}): if $X$ has finite entropy, then for $t>0$

$$ \dt{} \Ent(X_t) =  I(X_t) \quad ,\quad \dt{I(X_t)} = -2I(X_t) -2K(X_t) = -M(X_t) -K(X_t).  $$
As a consequence, a linear inequality on the information can be integrated along the semi-group to get the same inequality for the entropy. We also get that $\dt{I(X_t)} \leq -2I(X_t)$ which implies that $I(X_t) \leq e^{-2t}I(X_0)$. Moreover, the control role of $L$ comes from the observation that 

\begin{equation}{\label{eq:deriv_I_exp}}
    e^{2t}\dt{(e^{-2t}I(X_t))} = -2M(X_t).
\end{equation}
Now we are in position to prove Theorem \ref{thm:qbs}, from which Theorem \ref{thm:qss} will be an immediate corollary.

Let $X_0$ and $X_1$ be two isotropic random log-concave vectors in $\R^d$ and $\lambda \in [0,1]$, and $X_\lambda =\sqrt{1-\lambda}X_0 + \sqrt\lambda X_1$. Denote by  $(X_0)_t$, $(X_1)_t$ and $(X_\lambda)_t$ their evolution along the Ornstein-Uhlenbeck semi-group. Further define:
$$I_0(t) = I((X_0)_t) \quad , \quad I_1(t) = I((X_1)_t) \quad , \quad I_\lambda(t) = I((X_\lambda)_t), $$
and similarly, define $K_0(t), K_1(t), K_\lambda(t),M_0(t),M_1(t),M_\lambda(t) $; using (\ref{eq:deriv_I_exp}), the commutation property (\ref{eq:commute}), the observation that the Poincaré constants only decrease along the semi-group and Lemma \ref{lem:L2}, we get the following:
\begin{lem}\label{lem:last_lem} With the previous notations, for all $t\geq 0$,
\begin{equation*}
    \begin{split}
     -\dt{}((1-\lambda)e^{-2t}I_0(t) + & \lambda e^{-2t}I_1(t) - e^{-2t}I_\lambda(t)) \geq \frac{\lambda(1-\lambda)e^{-2t}}{\max(c_0(t),c_1(t))}(I_0(t)+K_0(t)+I_1(t)+K_1(t))  \\
    & \geq -\frac{\lambda(1-\lambda)}{2\max (c_0,c_1)}e^{-2t}\dt{}(I_0(t) + I_1(t)) \\
        \end{split}
\end{equation*}
\end{lem}
\comment{
\begin{proof}
\begin{align*}
 -\dt{}((1-\lambda)e^{-2t}I_0(t) + \lambda e^{-2t}I_1(t) - e^{-2t}I_\lambda(t)) & = 2e^{-2t}\left[(1-\lambda)L_0(t) + \lambdaM_1(t) -M_\lambda(t)\right] \\
 & \geq \frac{\lambda(1-\lambda)e^{-2t}}{\max(c_0(t),c_1(t)}(I_0(t)+K_0(t)+I_1(t)+K_1(t))\\
& \geq -\frac{\lambda(1-\lambda)e^{-2t}}{2\max (c_0,c_1)}\dt{}(I_0(t) + I_1(t)) \\
\end{align*}
\end{proof}}

\vspace{0.5cm}
\begin{proof}[Proofs of Theorem \ref{thm:qss} and Theorem \ref{thm:qbs}.]
~ \vspace{0.1cm}\\ 
Integrating the inequality of Lemma \ref{lem:last_lem} from $0$ to $\infty$, we get:
\[
\begin{split}
(1-\lambda ) I ( X_0 ) 
& + \lambda I ( X_1 ) - I ( X_\lambda )  \\ 
& \geq - \frac {\lambda(1-\lambda)}{2\max( c_0,c_1)}  \int_0^{+\infty} \e^{-2t} \frac d{dt} 
\left( I_0(t) + I_1(t) \right) \, dt \\
& = \frac {\lambda(1-\lambda)}{2\max( c_0,c_1)} \left(  I (X_0) + I (X_1) 
- 2 \int_0^\infty \e^{-2t} ( I_0(t) + I_1(t)) \, dt  \right)\\
& \geq \frac {\lambda(1-\lambda)}{ 4 \max( C_0,C_1)} \left( I ( X_0 ) + I ( X_1  ) \right)
\end{split}
\]
where in the last inequality, we used the fact that $I_1(t) \leq e^{-2t}I(X_1)$ and $I_0(t) \leq e^{-2t}I(X_0)$. This proves Theorem \ref{thm:qbs}.

Integrating Theorem \ref{thm:qbs} along the semi-group yields Theorem \ref{thm:qss}.
\end{proof}




\bibliography{biblio.bib}
\bibliographystyle{unsrt}

\end{document}